\documentclass[11pt,a4paper]{article}

\usepackage{amsmath,amssymb}
\usepackage{mathrsfs,amsfonts}
\usepackage{bm}
\usepackage{graphicx}
\usepackage{epsfig}
\usepackage{indentfirst}
\usepackage{amsthm}
\usepackage{amsfonts}
\numberwithin{equation}{section}

\newtheorem{proposition}{\textbf{Proposition}}[section]
\newtheorem{definition}{\textbf{Definition}}[section]
\newtheorem{theorem}{\textbf{Theorem}}[section]
\newtheorem{lemma}{\textbf{Lemma}}[section]

\newtheorem{remark}{\textbf{Remark}}[section]
\newtheorem{claim}{\textbf{Claim}}[section]

\begin{document}


\title{\textbf{Blow-up for a semilinear heat equation with Fujita's critical exponent on locally finite graphs}}
\author{ Yiting Wu}

\date{}
\maketitle

\vspace{-12pt}


\begin{minipage}{145mm}
\noindent{\small\textbf{Abstract} \;Let $G=(V,E)$ be a locally
finite, connected and weighted graph. We prove that, for a graph
satisfying curvature dimension condition $CDE'(n,0)$ and uniform
polynomial volume growth of degree $m$, all non-negative solutions
of the equation $\partial_tu=\Delta u+u^{1+\alpha}$ blow up in a
finite time provided that $\alpha=\frac{2}{m}$. We also consider
the blow-up problem under certain conditions for volume growth and
initial value. The obtained results provide a significant
complement to the work by Lin and Wu in earlier paper.
\newline\textbf{Keywords}: Semilinear heat equation; Critical exponent; Blow-up; Heat kernel estimate; Graph
\newline\noindent \textbf{2010 Mathematics Subject Classification}: 35A01; 35K91; 35R02; 58J35
}
\end{minipage}

\vspace{12pt}

\section{Introduction and main results}

In Euclidean space $\mathbb{R}^m$, the following semilinear heat equation
\begin{equation*}
\label{E1}
\partial_t u=\Delta u + u^{1+\alpha}, \quad\quad  u(0,x)=a(x)
\end{equation*}
has been considered by many authors (see \cite{Fujita,Hayakawa,Kobayashi,W1,W2}),
one of the most celebrated results is from Fujita \cite{Fujita}.
Fujita showed that all non-negative solutions of the above equation blow up in a finite time if $0<\alpha<\frac{2}{m}$,
whereas there is a global solution to the above equation for a sufficiently small initial value if $\alpha>\frac{2}{m}$.
The value $\alpha=\frac{2}{m}$ is called Fujita's critical exponent.
Several authors independently showed that the critical exponent belongs to the blow-up case in $\mathbb{R}^m$
(see \cite{Hayakawa,Kobayashi}).

In recent years the study of equations on graphs has received a lot of attention from many researchers,
see \cite{CLC,gly1,gly2,KS,LW2,LW3,W} and references cited therein.
In \cite{LW2} Lin and Wu extended Fujita's results to the case of locally finite graphs,
they studied the existence and nonexistence of global solutions for the following semilinear heat equation
\begin{equation}
\label{E:1.1}
\left\{
\begin{array}{lc}
\partial_t u=\Delta u + u^{1+\alpha}, &\, \text{$(t,x) \in$ $(0,+\infty)\times V$,}\\
u(0,x)=a(x), &\, \text{$x \in$ $V$,}
\end{array}
\right.
\end{equation}
where $\alpha$ is a positive parameter,  $a(x)$ is bounded, non-negative and non-trivial in $V$.

Under the conditions of curvature dimension $CDE'(n,0)$ and
uniform polynomial volume growth of degree $m$, Lin and Wu
\cite{LW2} proved that if $0<\alpha<\frac{2}{m}$, then all
non-negative solutions of \eqref{E:1.1} blow up in a finite time,
and if $\alpha>\frac{2}{m}$, then there exists a non-negative
global solution to \eqref{E:1.1} for any non-negative initial
value dominated by a sufficiently small Gaussian. Actually the
results of \cite{LW2} do not include the case of the critical
exponent $\alpha=\frac{2}{m}$. Recently, Wu \cite{W} proved the
blow-up result for $\alpha=\frac{2}{m}$ on finite graphs.
But,
the question of whether $\alpha=\frac{2}{m}$  still belongs to the
blow-up case on locally finite graphs remains unsolved.

In this paper, we deal with the blow-up problem for \eqref{E:1.1}
on locally finite graphs. We first consider the blow-up phenomenon
of \eqref{E:1.1} for $\alpha=\frac{2}{m}$ on graphs satisfying
uniform polynomial volume growth of degree $m$. And then we relax
volume growth condition and restrain initial value condition of
\eqref{E:1.1} to show the blow-up result for $0<\alpha\leq
\frac{2}{m}$. Our main results are as follows:


\begin{theorem}
\label{theorem2} Suppose that $D_\mu,D_\omega<\infty$ and that $G$
satisfies curvature condition $CDE'(n,0)$ and the volume growth
condition (c.1) below. If $\alpha=\frac{2}{m}$, then all
non-negative solutions of \eqref{E:1.1} blow up in a finite time.
\end{theorem}

\noindent Condition (c.1). There exist $c>0$ and $x_0\in V$ with
$a(x_0)>0$ such that the inequality
\begin{equation*}
c^{-1}r^m\leq \mathcal{V}(x_0,r) \leq cr^m \quad\quad (m>0)
\end{equation*}
holds for all large enough $r\geq r_0>0$. The condition (c.1) is
called uniform polynomial volume growth of degree $m$.

\begin{theorem}
\label{theorem3} Suppose that $D_\mu,D_\omega<\infty$ and that $G$
satisfies curvature condition $CDE'(n,0)$ and the volume growth
condition (c.2) below. If $0<\alpha\leq \frac{2}{m}$ and
$\inf_{x\in V} a(x)=a_0>0$, then all non-negative solutions of
\eqref{E:1.1} blow up in a finite time.
\end{theorem}

\noindent Condition (c.2). There exist $c',c''>0$ and $x_0\in V$
such that the inequality
\begin{equation*}
c'r^m \log^{-\zeta} r \leq \mathcal{V}(x_0,r)\leq c''r^m \log^\eta r \quad\quad (m>0,\;\; \zeta,\eta\geq 0 )
\end{equation*}
holds for all large enough $r\geq r_0>1$.


The remaining parts of this paper are organized as follows.
In Section 2, we introduce some notations
and related results on graphs which are essential to prove our main results.
In Section 3, we give an auxiliary result on
the representation of the solution in an integral equation in terms of its fundamental solution.
In Sections 4 and 5, we prove Theorems \ref{theorem2} and \ref{theorem3}, respectively.
Finally, in Section 6 we discuss a special case of Theorem \ref{theorem3} and illustrate that our results would provide a complement to earlier work on this topic.

\section{Preliminaries}

\subsection{Weighted graphs}

Let $G=(V,E)$ be a locally finite, connected graph without loops or multiple edges.
Here $V$ is the vertex set  and $E$ is the edge set that can be viewed as a symmetric subset of $V\times V$.
We write $y\sim x$ if an edge connects vertices $x$ and $y$.

We allow the edges on the graph to be weighted.
Let $\omega: V \times V\rightarrow[0,\infty)$ be an edge weight function that satisfies $\omega_{xy}=\omega_{yx}$ for all $x,y \in V$
and $\omega_{xy}>0$ if and only if $x\sim y$.
Moreover, we write $m(x):=\sum_{y\sim x}\omega_{xy}$ and assume that this weight function satisfies
$\omega_{\min}:=\inf_{(x,y)\in E} \omega_{xy}>0$.
Let $\mu: V\rightarrow (0,\infty)$ be a positive measure on $V$ and $\mu_{\max}:=\sup_{x\in V} \mu(x)$.
Given a weight and a measure, we define
$$D_\omega:=\frac{\mu_{\max}}{\omega_{\min}}$$
and
$$D_\mu:=\sup_{x\in V}\frac{m(x)}{\mu(x)}.$$

For any $x,y\in V$, $d(x,y)$, the distance between $x$ and $y$,  is the number of edges in the shortest path connecting $x$ and $y$.
The connectedness of $G$ ensures that $d(x,y)<\infty$ for any two points.

For a given vertex $x_0\in V$, let $B_r=B_r(x_0)$ denote the connected subgraph of $G$
consisting of those vertices in $G$ that are at most distance $r$ from $x_0$ and all edges of $G$ that such vertices span.
We can prove by induction that $\#B_r<\infty$.
Further, the volume of $B_r$ can be written as $\mathcal{V}(x_0,r)$,
which is defined by $\mathcal{V}(x_0,r)=\sum_{y\in B_r} \mu(y)$.

\subsection{Laplace operators on graphs}

We denote by $C(V)$ the set of real-valued functions on $V$, by
$$\ell^p(V,\mu)=\left\{ f\in C(V):\sum_{x \in V} \mu(x)|f(x)|^p<\infty \right\}, \;\; 1\leq p<\infty,$$
the set of $\ell^p$ integrable functions on $V$ with respect to
the measure $\mu$.  For $p=\infty$, let
$$\ell^{\infty}(V,\mu)=\left\{ f\in C(V):\sup_{x\in V}|f(x)|<\infty \right\}$$
be the set of bounded functions.

We define the $\mu$-Laplacian $\Delta:C(V)\rightarrow C(V)$ by, for any $x\in V$,
$$\Delta f(x)=\frac{1}{\mu(x)}\sum_{y\sim x}\omega_{xy}\left(f(y)-f(x)\right).$$
It is well-known that
$D_\mu<\infty$ is equivalent to the $\mu$-Laplacian $\Delta$ being bounded on $\ell^p(V,\mu)$ for all $p\in [1,\infty]$
(see \cite{HAESELER}).

Let $C(B_r,\partial B_r)=\left\{ f\in C(B_r):f|_{\partial B_r}=0 \right\}$ denote the set of functions on $B_r$
which vanish on $\partial B_r$,
where $\partial B_r$ is the boundary of $B_r$, i.e.,
$\partial B_r=\left\{ x\in B_r: \text{there is $y\sim x$ such that $y\notin B_r$}  \right\}.$
And the interior of $B_r$ is denoted by $\mathring{B_r}=B_r\setminus \partial B_r$.
The Dirichlet Laplacian $\Delta_r$ is defined as
\begin{equation*}
\Delta_r f=
\left\{
\begin{array}{lcc}
\Delta f&\quad \text{for}\; x\in \mathring{B_r} \\
0&\quad \text{otherwise}
\end{array}
\right.
\end{equation*}
for any $f\in C(B_r,\partial B_r)$.

\subsection{The heat kernel on graphs}

\begin{definition}[\cite{RW}]
\textnormal{
The heat kernel $p_r$ of $\Delta_r$ with Dirichlet boundary conditions is given by
$$p_r(t,x,y)=e^{-t\Delta_r}\delta_x(y) \;\;\;\;\; \text{for $x,y\in B_r$ and $t\geq 0$}.$$}
\end{definition}

The maximum principle implies the monotonicity of the heat kernel, that is, $p_{r}(t,x,y)\leq p_{r+1}(t,x,y)$.
Moreover, $p_r(t,x,y)$ satisfies the following properties (see \cite{RW}):
\begin{proposition}
\label{pr}
\textnormal{For $t,s>0$ and $x,y\in B_r$, we have\\
(i) $\partial_t p_r(t,x,y)=\Delta_r p(t,x,y)$, where $\Delta_r$ denotes the Dirichlet Laplacian applied in either $x$ or $y$;\\
(ii) $p_r(t,x,y)=0$ if $x\in \partial B_r$ or $y\in \partial B_r$;\\
(iii) $p_r(t,x,y)=p_r(t,y,x)$;\\
(iv) $p_r(t,x,y)\geq 0$;\\
(v) $\sum_{y\in B_r}\mu(y)p_r(t,x,y)\leq 1$;\\
(vi) $\sum_{z\in B_r}\mu(z)p_r(t,x,z)p_r(s,z,y)=p_r(t+s,x,y)$.}
\end{proposition}

The heat kernel for an infinite graph can be constructed by taking
an exhaustion of the graph. An exhaustion of $G$ is a sequence of
finite subsets with $U_r\subset U_{r+1}$ and $\bigcup_{r\in
\mathbb{N}}U_r=V$ (see \cite{AW,RW}). The connectedness of $G$
implies $\bigcup_{r\in \mathbb{N}}B_r=V$, so we can take
$U_r=B_r$.

\begin{definition}[\cite{AW,RW}]
\label{heat kernel}
\textnormal{For any $t>0$, $x,y\in V$, we define
\begin{equation}
\label{dingyi-p}
p(t,x,y)=\lim_{r\rightarrow \infty}p_r(t,x,y).
\end{equation}}
\end{definition}

Dini's Theorem implies that the convergence of \eqref{dingyi-p} is uniform in $t$ on every compact subset of $[0,\infty)$.
Also, it is known that $p(t,x,y)$ is differentiable in $t$, satisfies the heat equation $\partial_t u=\Delta u$,
and has the following properties (see \cite{AW,RW}):
\begin{proposition}
\textnormal{For $t,s>0$ and any $x,y\in V$, we have\\
(i) $\partial_t p(t,x,y)=\Delta p(t,x,y)$, where $\Delta$ denotes the $\mu$-Laplacian applied in either $x$ or $y$;\\
(ii) $p(t,x,y)=p(t,y,x)$;\\
(iii) $p(t,x,y)> 0$;\\
(iv)  $\sum_{y\in V}\mu(y)p(t,x,y)\leq 1$;\\
(v) $\sum_{z\in V}\mu(z)p(t,x,z)p(s,z,y)=p(t+s,x,y)$;\\
(vi) $p(t,x,y)$ is independent of the exhaustion used to define it.}
\end{proposition}

\begin{definition}
\textnormal{A graph $G$ is stochastically complete if
$$\sum_{y\in V}\mu(y)p(t,x,y)=1$$
for any $t>0$ and $x\in V$.}
\end{definition}

The heat kernel would generate a bounded solution of the heat equation on $G$ for any bounded initial condition.
Precisely, for any bounded function $u_0\in C(V)$, the function
$$u(t,x)=\sum_{y\in V}\mu(y)p(t,x,y)u_0(y) \quad (t>0, \,\, x\in V)$$
is bounded and differentiable in $t$, satisfies $\partial_tu=\Delta u$
and $\lim_{t\rightarrow0^+}u(t,x)=u_0(x)$ for any $x\in V$.

We define
$$P_t f(x)=\sum_{y\in V}\mu(y)p(t,x,y)f(y)$$
for any bounded function $f\in C(V)$,
then $P_t f(x)$ is a bounded solution of the heat equation and $P_tf(x)$ converges uniformly in $[0,T]$ for any $T>0$ (see \cite{LW2}).

$P_t$ is called the heat semigroup of the Laplace operator.
Further, the different definitions of the heat semigroup coincide when $\Delta$ is a bounded operator, that is,
$$P_tf(x)=e^{t\Delta}f(x)=\sum_{y\in V}\mu(y)p(t,x,y)f(y).$$
In the following proposition, we transcribe some useful properties of the heat semigroup $P_t$.
\begin{proposition}[\cite{HLLY}]
\textnormal{For any bounded function $f\in C(V)$ and $t,s>0$, $x\in V$, we have\\
(i) $P_tP_sf(x)=P_{t+s}f(x)$;\\
(ii) $\Delta P_tf(x)=P_t\Delta f(x)$;\\
(iii) $\lim_{t\rightarrow 0^+}P_t f(x)=P_0f(x)=f(x)$.}
\end{proposition}

\subsection{Curvature dimension condition and Gaussian heat kernel estimate}

We recall the definition of two natural bilinear forms associated with the $\mu$-Laplacian.
\begin{definition}[\cite{BL}]
\textnormal{The gradient form $\Gamma$ is defined by
\begin{equation*}
\begin{split}
2\Gamma(f,g)(x)=&(\Delta(fg)-f\Delta(g)-\Delta(f)g)(x)\\
=&\frac{1}{\mu(x)}\sum_{y\sim x}\omega_{xy}(f(y)-f(x))(g(y)-g(x)), \quad\quad f,g\in C(V).
\end{split}
\end{equation*}
The iterated gradient form $\Gamma_2$ is defined by
$$2\Gamma_2(f,g)(x)=(\Delta\Gamma(f,g)-\Gamma(f,\Delta g)-\Gamma(\Delta f,g))(x), \quad\quad f,g\in C(V).$$
For simplicity, we write $\Gamma(f)=\Gamma(f,f)$ and $\Gamma_2(f)=\Gamma_2(f,f).$
}
\end{definition}

Using the bilinear forms above one can define the following exponential curvature dimension inequality.
\begin{definition}[\cite{BHLLMY}]
\textnormal{We say that a graph $G$ satisfies the exponential curvature dimension inequality $CDE'(x,n,K)$
if for any positive function $f\in C(V)$, we have
$$\Gamma_2(f)(x)-\Gamma\Bigg(f,\frac{\Gamma(f)}{f}\Bigg)(x)\geq \frac{1}{n}f(x)^2(\Delta \log f)(x)^2+K\Gamma(f)(x).$$
We say that $CDE'(n,K)$ is satisfied if $CDE'(x,n,K)$ is satisfied for all $x\in V$.}
\end{definition}

With the help of curvature dimension condition $CDE'(n,0)$, Horn et al.\cite{HLLY} derived the Gaussian heat kernel estimate, as follows:
\begin{proposition}
\label{estimate}
\textnormal{Suppose that $D_\mu, D_\omega<\infty$ and $G$ satisfies $CDE'(n,0)$.
Then there exists a positive constant $c_1$
such that, for any $x,y \in V$ and $t>0$,
\begin{equation*}
p(t,x,y)\leq \frac{c_1}{\mathcal{V}(x,\sqrt{t})}.
\end{equation*}
Furthermore, for any $t_0>0$, there exist positive constants $c_2$ and $c_3$ such that
\begin{equation*}
p(t,x,y)\geq \frac{c_2}{\mathcal{V}(x,\sqrt{t})}\exp\left(-c_3\frac{d^2(x,y)}{t}\right)
\end{equation*}
for all $x,y\in V$ and $t>t_0$, where $c_1$ depends on $n,D_\omega,D_\mu$ and is denoted by $c_1=c_1(n,D_\omega,D_\mu)$,
$c_2$ depends on $n$ and is denoted by $c_2=c_2(n)$,\,
$c_3$ depends on $n,D_\omega,D_\mu$ and is denoted by $c_3=c_3(n,D_\omega,D_\mu)$.
In particular, for any $t_0>0$, we have
\begin{equation*}
p(t,x,x)\geq \frac{c_2(n)}{\mathcal{V}(x,\sqrt{t})}
\end{equation*}
for all $x\in V$ and $t>t_0$,
where $c_2(n)$ is positive.}
\end{proposition}

\section{Auxiliary result}


In order to prove our main results, we need an auxiliary result, which provides a representation to the solution of \eqref{E:1.1} via
an integral equation with the heat kernel. Firstly,
we introduce the definition of the solution for equation \eqref{E:1.1} on graphs.

\begin{definition}
\label{definition:def1.1}
\textnormal
{Let $T$ be a positive number. A non-negative function $u=u(t,x)$ is said to be a non-negative solution of \eqref{E:1.1} in $[0,T]$,
if $u$ is continuous with respect to $t$ and satisfies \eqref{E:1.1} in $[0,T]\times V$.}
\end{definition}

Let $T_\infty$ denote the supremum of all $T$ satisfying the condition that
for any $T'<T$ the solution $u(t,x)$ of \eqref{E:1.1} is bounded in $[0,T']\times V$.

\begin{definition}
\label{definition:def1.2}
\textnormal
{A solution $u(t,x)$ of \eqref{E:1.1} is said to blow up in a finite time, provided that $T_\infty<\infty$.
A solution $u(t,x)$ of \eqref{E:1.1} is global if $T_\infty=\infty$.}
\end{definition}


Next, we employ the heat kernel to represent the solution of \eqref{E:1.1}.

\begin{theorem}
\label{theorem1}
\textnormal{Suppose that $D_\mu<\infty$.
For any $T>0$, the bounded and non-negative solution $u$ of \eqref{E:1.1} satisfies
\begin{equation}
\label{th31}
u(t,x)=P_t a(x)+\int_0^t P_{t-s}u(s,x)^{1+\alpha}ds \quad\quad
(0<t<T, \; x\in V).
\end{equation}
}
\end{theorem}


\begin{lemma}[\cite{g1,RW}]
\label{lemma:lemma3.1}
\textnormal{The following statements are equivalent:\\
(i) The graph $G$ is stochastically incomplete;\\
(ii) There exists a non-zero, bounded function $u$ satisfying
\begin{equation}
\label{E:3.0}
\left\{
\begin{array}{lc}
\partial_tu=\Delta u, &\, \text{$(t,x) \in$ $(0,T)\times V$},\\
u(0,x)=0, &\, \text{$x \in$ $V$}.
\end{array}
\right.
\end{equation}
}
\end{lemma}

\begin{lemma}[\cite{KL}]
\label{lemma:lemma3.2}
\textnormal{The graph $G$ is stochastically complete if $D_\mu<\infty$.}
\end{lemma}


\begin{proof}[Proof of Theorem \ref{theorem1}.]
We consider the following Cauchy problem
\begin{equation}
\label{E:3.1}
\left\{
\begin{array}{lc}
\partial_t v=\Delta v + g(t,x), &\, \text{$(t,x) \in$ $(0,T)\times V$,}\\
v(0,x)=a(x), &\, \text{$x \in$ $V$,}
\end{array}
\right.
\end{equation}
where $g$ is bounded and continuous with respect to $t$ in $[0,T)$.

Set
$$v_1(t,x)=\int_0^t P_{t-s}g(s,x)ds \quad\quad (0\leq t<T, \; x\in V).$$

Since $P_{t-s}g(s,x)$ and $\partial_tP_{t-s}g(s,x)$ are continuous with respect to $t$ and $s$.
Differentiating $v_1(t,x)$ with respect to $t$ gives
\begin{equation*}
\begin{split}
\partial_tv_1(t,x)&=\int_0^{t}\partial_t P_{t-s}g(s,x)ds+P_0 g(t,x)\\
&=\int_0^{t} \Delta P_{t-s}g(s,x)ds+P_0 g(t,x)\\
&=\Delta v_1(t,x)+g(t,x).
\end{split}
\end{equation*}

Letting $v(t,x)=P_t a(x)+v_1(t,x)$, it follows that for any $t\in (0,T)$ and $x\in V$,
\begin{equation*}
\begin{split}
\partial_t v(t,x)=\partial_tP_t a(x)+\partial_tv_1(t,x)=\Delta P_t a(x)+\Delta v_1(t,x)+g(t,x)=\Delta v(t,x)+g(t,x),
\end{split}
\end{equation*}
and $v(0,x)=a(x)$.
Thus, $v(t,x)$ is a solution of \eqref{E:3.1}.

On the other hand, by the assumption $D_\mu<\infty$ in Theorem \ref{theorem1}, we deduce from
Lemma \ref{lemma:lemma3.1} and Lemma \ref{lemma:lemma3.2} that
the equation \eqref{E:3.0} has only zero solution,
which implies that the solution of \eqref{E:3.1} is unique
and satisfies
\begin{equation}
\label{E:3.13}
v(t,x)=P_t a(x)+v_1(t,x)=P_t a(x)+\int_0^t P_{t-s}g(s,x)ds \quad\quad (0<t<T, \; x\in V).
\end{equation}

For an arbitrary bounded and non-negative solution $u(t,x)$ of \eqref{E:1.1},
we take $g(t,x)=u(t,x)^{1+\alpha}$ in \eqref{E:3.1}.
By \eqref{E:3.13}, we deduce that
$P_t a(x)+\int_0^t P_{t-s}u(s,x)^{1+\alpha}ds$
is the unique solution of \eqref{E:3.1} for $g(t,x)=u(t,x)^{1+\alpha}$.

In view of $\partial_t u=\Delta u+u^{1+\alpha}$, so $u(t,x)$ is a solution of \eqref{E:3.1} for $g(t,x)=u(t,x)^{1+\alpha}$.
Therefore, we have for any $(t,x)\in (0,T)\times V$,
\begin{equation*}
\begin{split}
u(t,x)&=P_t a(x)+\int_0^t P_{t-s}u(s,x)^{1+\alpha}ds.
\end{split}
\end{equation*}

This completes the proof of Theorem \ref{theorem1}.
\end{proof}

\section{Proof of Theorem \ref{theorem2}}

Since $a(x)$ is non-trivial in $V$, we may assume $a(x_0)>0$ with $x_0\in V$.
In what follows we fix the vertex $x_0\in V$ and let $B_r=B_r(x_0)$.

We first prove two lemmas below.


\begin{lemma}
\label{lemma:lemma4.1}
\textnormal{Suppose that $D_\mu<\infty$. If the equation \eqref{E:1.1} has a non-negative global solution $u(t,x)$,
then there exists a constant $C'$ such that
$t^{\frac{1}{\alpha}}P_t a(x)\leq C'$ for any $t>0$ and $x\in V$.}
\end{lemma}


\begin{proof}[Proof]
From Lemma \ref{lemma:lemma3.2} it follows that the graph is stochastically complete, that is,
$$\sum_{y\in V}\mu(y)p(t,x,y)=1 \quad\quad  (t>0, \; x\in V).$$

Applying the Jensen's inequality, we obtain for any $t>s>0$ and $x\in V$,
\begin{equation}
\label{le4.1:0}
\begin{split}
P_{t-s}u(s,x)^{1+\alpha}&=\sum_{y\in V}\mu(y)p(t-s,x,y)u(s,y)^{1+\alpha}\\
&\geq \left( \sum_{y\in V}\mu(y)p(t-s,x,y)u(s,y) \right)^{1+\alpha}\\
&=\left(P_{t-s}u(s,x)\right)^{1+\alpha}.
\end{split}
\end{equation}

Note that $P_ta(x)\geq 0$, using Theorem \ref{theorem1} along with \eqref{le4.1:0}, one has
\begin{equation}
\label{le4.1:1}
u(t,x)\geq \int_0^tP_{t-s}u(s,x)^{1+\alpha}ds\geq \int_0^t \left(P_{t-s}u(s,x)\right)^{1+\alpha}ds \quad\quad (t>0, \; x\in V).
\end{equation}

Again, from Theorem \ref{theorem1} and the fact that $u$ is non-negative we get $u(t,x)\geq P_ta(x)$, which, together with \eqref{le4.1:1}, yields
\begin{equation}
\label{le4.1:2}
u(t,x)\geq \int_0^t \left(P_{t-s}P_sa(x)\right)^{1+\alpha}ds = \int_0^t \left(P_t a(x)\right)^{1+\alpha}ds
=t\left(P_t a(x)\right)^{1+\alpha} \quad (t>0, \; x\in V).
\end{equation}

Further, utilizing \eqref{le4.1:2} and the Jensen's inequality gives
\begin{equation}
\label{le4.1:3}
\begin{split}
P_{t-s}u(s,x)&\geq sP_{t-s}\left(P_s a(x)\right)^{1+\alpha}\\
&=s\sum_{y\in V}\mu(y)p(t-s,x,y)\left(P_s a(x)\right)^{1+\alpha}\\
&\geq s \left( \sum_{y\in V}\mu(y)p(t-s,x,y)\left(P_s a(x)\right)\right)^{1+\alpha}\\
&=s \left( P_{t-s}P_s a(x)\right)^{1+\alpha}\\
&=s \left( P_t a(x)\right)^{1+\alpha}.
\end{split}
\end{equation}

Substituting \eqref{le4.1:3} into \eqref{le4.1:1}, we get
\begin{equation}
\label{le4.1:4}
u(t,x)\geq \int_0^t s^{1+\alpha}\left( P_t a(x)\right)^{{(1+\alpha)}^2} ds
=\frac{1}{2+\alpha}t^{2+\alpha}\left( P_t a(x)\right)^{{(1+\alpha)}^2} \quad\quad (t>0, \; x\in V).
\end{equation}

It follows from \eqref{le4.1:4} and the Jensen's inequality that
\begin{equation}\label{le4.1:41}
\begin{split}
P_{t-s}u(s,x)&\geq\frac{1}{2+\alpha}s^{2+\alpha}P_{t-s}\left( P_s
a(x)\right)^{{(1+\alpha)}^2} \geq
\frac{1}{2+\alpha}s^{2+\alpha}\left( P_t
a(x)\right)^{{(1+\alpha)}^2}.
\end{split}
\end{equation}

Substituting \eqref{le4.1:41} into \eqref{le4.1:1}, one finds
\begin{equation*}
u(t,x)\geq
\frac{t^{1+(1+\alpha)+(1+\alpha)^2}}{(1+(1+\alpha)+(1+\alpha)^2)
(1+(1+\alpha))^{1+\alpha}}\left( P_t a(x)\right)^{{(1+\alpha)}^3}.
\end{equation*}

Repeatedly, performing the substitution in \eqref{le4.1:1} for $k$
times, we obtain that for any $t>0$ and $x\in V$,
\begin{equation*}
\begin{split}
&u(t,x)\\
\geq&\frac{t^{1+(1+\alpha)+(1+\alpha)^2+\cdots +(1+\alpha)^{k-1}}\left( P_t a(x)\right)^{{(1+\alpha)}^k}}
{\left(1+(1+\alpha)+\cdots +(1+\alpha)^{k-1}\right) \cdots\left(1+(1+\alpha)+(1+\alpha)^2\right)^{{(1+\alpha)}^{k-3}}
\left(1+(1+\alpha)\right)^{{(1+\alpha)}^{k-2}}}\\
=&\frac{t^{\frac{(1+\alpha)^k-1}{\alpha}}\left( P_t a(x)\right)^{{(1+\alpha)}^k}}
{\left( \frac{(1+\alpha)^k-1}{\alpha} \right)\cdots \left(\frac{(1+\alpha)^3-1}{\alpha} \right)^{(1+\alpha)^{k-3}}
\left(\frac{(1+\alpha)^2-1}{\alpha} \right)^{(1+\alpha)^{k-2}}}.
\end{split}
\end{equation*}

Hence,
\begin{equation*}
\begin{split}
t^{\frac{(1+\alpha)^k-1}{\alpha(1+\alpha)^k}} P_t a(x)\leq u(t,x)^{\frac{1}{(1+\alpha)^k}}
\prod_{i=2}^{k}\left( \frac{(1+\alpha)^i-1}{\alpha} \right)^{(1+\alpha)^{-i}}.
\end{split}
\end{equation*}

Taking the limit $k\rightarrow \infty$ in both sides of the above
inequalities, we arrive at
\begin{equation}
\label{le4.1:6}
\begin{split}
t^{\frac{1}{\alpha}} P_t a(x)\leq
\prod_{i=2}^{\infty}\left( \frac{(1+\alpha)^i-1}{\alpha} \right)^{(1+\alpha)^{-i}}.
\end{split}
\end{equation}

Note that
\begin{equation}
\begin{split}
\label{le4.1:7}
\prod_{i=2}^{\infty}\left(
\frac{(1+\alpha)^i-1}{\alpha} \right)^{(1+\alpha)^{-i}}
&=\exp\Bigg\{\sum_{i=2}^\infty (1+\alpha)^{-i} \log\left(
\frac{(1+\alpha)^i-1}{\alpha} \right)\Bigg\}\\
&\leq\exp\Bigg\{\sum_{i=2}^\infty \frac{\log\left( i(1+\alpha)^i
\right)}{(1+\alpha)^{i}}\Bigg\}\\
&<\infty.
\end{split}
\end{equation}

Hence the assertion of Lemma  \ref{lemma:lemma4.1} follows from
\eqref{le4.1:6} and \eqref{le4.1:7}.
\end{proof}


\begin{lemma}
\label{lemma:lemma4.2}
\textnormal{Suppose that
$\alpha=\frac{2}{m}$ and $D_\mu,D_\omega<\infty$. Let $G$ satisfy
$CDE'(n,0)$ and $\mathcal{V}(x_0,t)\leq ct^m$ for $x_0\in V$,
$t\geq r_0$, where $c, m, r_0$ are some positive constants and
$a(x_0)>0$. If the equation \eqref{E:1.1} has a non-negative
global solution $u(t,x)$, then there exists a constant $C''$ such
that
\begin{equation}
\label{4.2-e11}
\begin{split}
\sum_{y\in B_r}\mu(y)a(y)\leq C''
\end{split}
\end{equation}
for any $r>\sqrt{\frac{\rho}{c_3}}$, where $\rho=\max\{t_0, r_0^2\}$.}
\end{lemma}

\begin{proof}[Proof]
Using Proposition \ref{estimate}, we have for $t>t_0>0$,
\begin{equation*}
\begin{split}
t^{\frac{m}{2}}P_t a(x_0)&=t^{\frac{m}{2}}\sum_{y\in V}\mu(y)p(t,x_0,y)a(y)\\
&\geq t^{\frac{m}{2}}\sum_{y\in V}\mu(y)a(y)
\frac{c_2}{\mathcal{V}(x_0,\sqrt{t})}\exp\left(-c_3\frac{d^2(x_0,y)}{t}\right).
\end{split}
\end{equation*}
In view of $\alpha=\frac{2}{m}$ and $\mathcal{V}(x_0,t)\leq ct^m$, a simple calculation shows that for
 $t>\rho=\max\{t_0, r_0^2\}$ and $r>0$,
\begin{equation*}
\begin{split}
t^{\frac{1}{\alpha}}P_t a(x_0)&= t^{\frac{m}{2}}P_t a(x_0)\\
&\geq \frac{c_2}{c}\sum_{y\in V}\mu(y)a(y)\exp\left(-c_3\frac{d^2(x_0,y)}{t}\right)\\
&\geq \frac{c_2}{c}\exp\left(-\frac{c_3r^2}{t}\right)\sum_{y\in B_r}\mu(y)a(y).
\end{split}
\end{equation*}

Taking $t=c_3r^2$ with a straightforward application of Lemma
\ref{lemma:lemma4.1}, we deduce that for any
$r>\sqrt{\frac{\rho}{c_3}}$,
\begin{equation*}
\begin{split}
\sum_{y\in B_r}\mu(y)a(y)\leq C'',
\end{split}
\end{equation*}
where $C''=e\frac{cC'}{c_2}$. This completes the proof of Lemma 4.2.
\end{proof}


\begin{remark}
\label{remark:remark4.1} \textnormal{Note that for any $t$ the
non-negative global solution $u(t,x)$ of \eqref{E:1.1} can be
regarded as the initial value of the equation \eqref{E:1.1}, so
the assertion of Lemma \ref{lemma:lemma4.2} implies that for any
$t>0$ and $r>\sqrt{\frac{\rho}{c_3}}$},
\begin{equation}
\label{4.2-e22}
\sum_{y\in B_r}\mu(y)u(t,y)\leq C''.
\end{equation}
\end{remark}


\begin{proof}[Proof of Theorem \ref{theorem2}.]

We shall prove Theorem \ref{theorem2} by contradiction. Assume
that there exists a non-negative global solution $u(t,x)$ of
\eqref{E:1.1}.

For any given $r>0$, we define
\begin{equation*}
G(t,r)=\sum_{y\in B_r}\mu(y)p(t,x_0,y)u(t,y).
\end{equation*}

By using the upper estimate of the heat kernel in Proposition
\ref{estimate} and the condition of Theorem \ref{theorem2},
 $\mathcal{V}(x_0,t)\geq c^{-1}t^m$, one has for $t\geq r_0^2$,
\begin{equation*}
\begin{split}
G(t,r)&\leq \sum_{y\in B_r}\mu(y)u(t,y)\frac{c_1}{\mathcal{V}(x_0,\sqrt{t})}\\
&\leq cc_1t^{-\frac{m}{2}}\sum_{y\in B_r}\mu(y)u(t,y).
\end{split}
\end{equation*}

Utilizing the claim of Remark \ref{remark:remark4.1}, we derive
from the above inequality that for $t\geq r_0^2$ and
$r>\sqrt{\frac{\rho}{c_3}}$,
\begin{equation}
\label{th4.1:1-0}
\begin{split}
t^{\frac{m}{2}}G(t,r)\leq cc_1C''.
\end{split}
\end{equation}

On the other hand, since $u(s,x)$ and $a(x)$ are non-negative, it
follows from Theorem \ref{theorem1} that
\begin{equation}
\label{th4.1:1} u(t,x)\geq P_ta(x)=\sum_{y\in
V}\mu(y)p(t,x,y)a(y)\geq C_1p(t,x_0,x),
\end{equation}
where $C_1=\mu(x_0)a(x_0)$, and

\begin{equation}
\label{th4.1:2}
\begin{split}
u(t,x)&\geq \int_0^t P_{t-s}u(s,x)^{1+\alpha}ds=\int_0^t
\sum_{y\in V}\mu(y)p(t-s,x,y)u(s,y)^{1+\alpha} ds.
\end{split}
\end{equation}

Additionally, a direct application of Proposition \ref{estimate}
with $\mathcal{V}(x_0,t)\leq ct^m$ gives
\begin{equation}
\label{th4.1:3}
p(t,x_0,x)\geq \frac{c_2}{\mathcal{V}(x_0,\sqrt{t})}\exp\left( -\frac{c_3d^2(x_0,x)}{t}\right)
\geq \frac{c_2}{c}t^{-\frac{m}{2}}\exp\left( -\frac{c_3d^2(x_0,x)}{t}\right)
\end{equation}
for $t>\rho$ and $x\in V$.

Combining \eqref{th4.1:2}, \eqref{th4.1:3} and \eqref{th4.1:1}, we
deduce that for $x\in V$, $t\geq c_3\alpha r^2$ and
$r>\sqrt{\frac{\rho}{c_3\alpha}}$,
\begin{equation}\label{th4.1:31}
\begin{split}
u(t,x)&\geq C_1^{1+\alpha}\int_0^t \sum_{y\in V}\mu(y)p(t-s,x,y)p(s,x_0,y)^{1+\alpha}ds\\
&> \frac{c_2^\alpha C_1^{1+\alpha}}{c^\alpha}\int_{c_3\alpha
r^2}^t s^{-\frac{m\alpha}{2}}
\sum_{y\in V}\mu(y)p(t-s,x,y)p(s,x_0,y)\exp\left( -\frac{c_3\alpha d^2(x_0,y)}{s}\right)ds\\
&\geq \frac{c_2^\alpha C_1^{1+\alpha}}{c^\alpha}\int_{c_3\alpha r^2}^t s^{-\frac{m\alpha}{2}}
\exp\left( -\frac{c_3\alpha r^2}{s}\right)
\sum_{y\in B_r}\mu(y)p(t-s,x,y)p(s,x_0,y)ds\\
&\geq C_2\int_{c_3\alpha r^2}^t s^{-\frac{m\alpha}{2}}
\sum_{y\in B_r}\mu(y)p(t-s,x,y)p(s,x_0,y)ds\\
&\geq C_2\int_{c_3\alpha r^2}^t s^{-\frac{m\alpha}{2}}
\sum_{y\in B_r}\mu(y)p_r(t-s,x,y)p_r(s,x_0,y)ds,
\end{split}
\end{equation}
where $C_2=\frac{c_2^\alpha C_1^{1+\alpha}}{ec^\alpha}$, the last
inequality follows from $p(t,x,y)\geq p_r(t,x,y)$ which is due to
the fact that $p_{r}(t,x,y)\leq p_{r+1}(t,x,y)$ and
$p(t,x,y)=\lim_{r\rightarrow\infty} p_r(t,x,y)$.

Applying $m\alpha=2$ and $\sum_{y\in
B_r}\mu(y)p_r(t-s,x,y)p_r(s,x_0,y)=p_r(t,x_0,x)$ to
\eqref{th4.1:31} gives
\begin{equation*}
\begin{split}
u(t,x)&\geq C_2p_r(t,x_0,x)\int_{c_3\alpha r^2}^t s^{-1}ds\\
&=C_2p_r(t,x_0,x)\log\left(\frac{t}{c_3\alpha r^2}\right)
\end{split}
\end{equation*}
for $x\in V$, $t\geq c_3\alpha r^2$ and $r>\sqrt{\frac{\rho}{c_3\alpha}}$.

Hence, we deduce that for any $t\geq c_3\alpha r^2$ and
$r>\sqrt{\frac{\rho}{c_3\alpha}}$,
\begin{equation}
\label{th4.1:4}
\begin{split}
G(t,r)&\geq \sum_{y\in B_r}\mu(y)p_r(t,x_0,y)u(t,y)\\
&\geq C_2\log\left(\frac{t}{c_3\alpha r^2}\right)\sum_{y\in B_r}\mu(y)p_r(t,x_0,y)p_r(t,x_0,y)\\
&=C_2\log\left(\frac{t}{c_3\alpha r^2}\right)p_r(2t,x_0,x_0).
\end{split}
\end{equation}

In addition, by Proposition \ref{estimate} and Definition
\ref{heat kernel}, there exists a constant
$R\geq\sqrt{\frac{\rho}{c_3\alpha}}$, which is independent of
$t$, such that for $r>R$,
\begin{equation}
\label{th4.1:5}
\begin{split}
p_r(2t,x_0,x_0)\geq \frac{c_2}{\mathcal{V}(x_0,\sqrt{2t})} \quad\quad (t>t_0).
\end{split}
\end{equation}

Combining \eqref{th4.1:4}, \eqref{th4.1:5} with
$\mathcal{V}(x_0,t)\leq ct^m$, we obtain
\begin{equation*}
\begin{split}
G(t,r)&\geq
\frac{C_2c_2}{c}(2t)^{-\frac{m}{2}}\log\left(\frac{t}{c_3\alpha
r^2}\right),
\end{split}
\end{equation*}
this yields
\begin{equation}
\label{th4.1:6}
\begin{split}
t^{\frac{m}{2}}G(t,r)&\geq C_3\log\left(\frac{t}{c_3\alpha r^2}\right)
\end{split}
\end{equation}
for any $t\geq c_3\alpha r^2$ and $r>R$, where $C_3=\frac{C_2c_2}{2^{\frac{m}{2}}c}$ and $R\geq\sqrt{\frac{\rho}{c_3\alpha}}$.

Consequently, from \eqref{th4.1:1-0} and \eqref{th4.1:6} we derive
that
\begin{equation}
\label{th4.1:7}
\begin{split}
C_3\log\left(\frac{t}{c_3\alpha r^2}\right)&\leq cc_1C''
\end{split}
\end{equation}
for $t\geq c_3\alpha r^2$ and $r>\max\left\{R,
\sqrt{\frac{\rho}{c_3}}\right\}$.

In fact, it is easy to observe that the reverse and strict
inequality of \eqref{th4.1:7} holds for sufficiently large $t$.
This leads to a contradiction. Hence, there is no non-negative
global solution for the equation \eqref{E:1.1}. The proof of
Theorem \ref{theorem2} is complete.

\end{proof}

\section{Proof of Theorem \ref{theorem3}}

We first introduce and prove the following lemma.


\begin{lemma}
\label{lemma:lemma5.2}
\textnormal{Suppose that $0<\alpha\leq
\frac{2}{m}$ and $D_\mu,D_\omega<\infty$. Let $G$ satisfy
$CDE'(n,0)$ and $\mathcal{V}(x_0,t)\leq c''t^m \log^\eta t$ for
$t\geq r_0>1$, where $c'', m, \eta$ are some positive constants
and $a(x_0)>0$. If the equation \eqref{E:1.1} has a non-negative
global solution $u(t,x)$, then there exists a constant $C''$ such that
\begin{equation}
\label{5.1-e33}
\begin{split}
\sum_{y\in B_r}\mu(y)u(t,y)\leq C''\left(\log\sqrt{c_3r^2}\right)^\eta
\end{split}
\end{equation}
for any $t>0$ and $r>\sqrt{\frac{\rho}{c_3}}$, where $\rho=\max\{t_0, r_0^2\}$.}
\end{lemma}

\begin{proof}[Proof]
Using Proposition \ref{estimate} with $\mathcal{V}(x_0,t)\leq
c''t^m \log^\eta t$, we have for $t>\rho$ and $r>0$,
\begin{equation*}
\begin{split}
t^{\frac{m}{2}}P_t a(x_0) &\geq
\frac{c_2}{c''}\exp\left(-\frac{c_3r^2}{t}\right)\left(\log
\sqrt{t}\right)^{-\eta}\sum_{y\in B_r}\mu(y)a(y),
\end{split}
\end{equation*}
which, along with $0<\alpha\leq \frac{2}{m}$ and the inequality
asserted by Lemma \ref{lemma:lemma4.1}, leads us to
\begin{equation*}
\begin{split}
\sum_{y\in B_r}\mu(y)a(y)\leq \frac{c''C'}{c_2}\exp\left(\frac{c_3r^2}{t}\right)\left(\log \sqrt{t}\right)^{\eta}.
\end{split}
\end{equation*}

Choosing $t=c_3r^2$ in the above inequality gives
\begin{equation}
\label{e5.1}
\begin{split}
\sum_{y\in B_r}\mu(y)a(y)\leq C''\left(\log
\sqrt{c_3r^2}\right)^{\eta},
\end{split}
\end{equation}
where $C''=e\frac{c''C'}{c_2}$ and $r>\sqrt{\frac{\rho}{c_3}}$.

Note that for any $t$ the non-negative global solution $u(t,x)$ of
\eqref{E:1.1} can be regarded as the initial value of the equation
\eqref{E:1.1}, thus inequality \eqref{e5.1} implies that

\begin{equation*}
\begin{split}
\sum_{y\in B_r}\mu(y)u(t,y)\leq C''\left(\log \sqrt{c_3r^2}\right)^{\eta}.
\end{split}
\end{equation*}

The Lemma \ref{lemma:lemma5.2} is proved.

\end{proof}


\begin{proof}[Proof of Theorem \ref{theorem3}.]
We shall now prove Theorem \ref{theorem3} by contradiction. Assume
that $u(t,x)$ is a non-negative global solution of \eqref{E:1.1}.

Let
\begin{equation*}
G(t,r)=\sum_{y\in B_r}\mu(y)p(t,x_0,y)u(t,y) \quad\quad (t>0,\;r>0).
\end{equation*}

By using Proposition \ref{estimate}, Lemma \ref{lemma:lemma5.2}
and $\mathcal{V}(x_0,t)\geq c't^m\log^{-\zeta}t$, we obtain that
for $t\geq r_0^2$ and $r>\sqrt{\frac{\rho}{c_3}}$,
\begin{equation}
\label{E:5.1-1}
\begin{split}
G(t,r)&\leq \frac{c_1}{c'}t^{-\frac{m}{2}}\left(\log\sqrt{t}\right)^{\zeta}\sum_{y\in B_r}\mu(y)u(t,y)\\
&\leq \frac{c_1C''}{2^\zeta c'}t^{-\frac{m}{2}}\left(\log t\right)^{\zeta}\left(\log\sqrt{c_3r^2}\right)^\eta.
\end{split}
\end{equation}

Applying Theorem \ref{theorem1} and the above-mentioned inequality
$p(t,x,y)\geq p_r(t,x,y)$, we deduce that for any $t>0$ and $r>0$,
\begin{equation}
\label{E:5.2}
\begin{split}
G(t,r)&\geq \sum_{y\in B_r}\mu(y)p(t,x_0,y)\int_0^t P_{t-s}u(s,y)^{1+\alpha} ds\\
&\geq \sum_{y\in B_r}\mu(y)p(t,x_0,y)\int_0^t \sum_{z\in V}\mu(z)p(t-s,y,z)u(s,z)^{1+\alpha} ds\\
&\geq \int_0^t \sum_{y\in B_r}\sum_{z\in B_r}\mu(y)\mu(z)p(t,x_0,y)p(t-s,y,z)u(s,z)^{1+\alpha} ds\\
&\geq \int_0^t \sum_{y\in B_r}\sum_{z\in B_r}\mu(y)\mu(z)p_r(t,x_0,y)p_r(t-s,y,z)u(s,z)^{1+\alpha} ds\\
&=\int_0^t \sum_{z\in B_r}\mu(z)p_r(2t-s,x_0,z)u(s,z)^{1+\alpha} ds\\
&\geq \int_0^t \left(\sum_{z\in B_r}\mu(z)p_r(2t-s,x_0,z)u(s,z)\right)^{1+\alpha} ds,
\end{split}
\end{equation}
where the last inequality follows from the inequality of
Proposition \ref{pr} $\sum_{z\in B_r}\mu(z)p_r(2t-s,x_0,z)\leq 1$
and the Jensen-type inequality
$\sum_{i=1}^n\lambda_i\varphi(\tau_i)\geq
\varphi(\sum_{i=1}^n\lambda_i\tau_i),$ where $\varphi$ is convex
and satisfies $\varphi(0)=0$, $\sum_{i=1}^n\lambda_i \leq 1$,
$\lambda_i\geq 0$, $i=1,2,\ldots, n$ (see \cite{LW3}).

Using inequality \eqref{le4.1:2} and $a_{0}=\inf_{x\in V}{a(x)}$,
we deduce that
\begin{equation*}
\begin{split}
u(s,z)&\geq s\left(\sum_{x\in V}\mu(x)p(s,z,x)a(x)\right)^{1+\alpha}\\
&\geq s\left(a_0\sum_{x\in V}\mu(x)p(s,z,x)\right)^{\alpha}\left(\sum_{x\in V}\mu(x)p(s,z,x)a(x)\right).\\
\end{split}
\end{equation*}
Further, it follows from the stochastic completeness of $G$ that
for any $s>0$, $z\in V$ and $r>0$,
\begin{equation}
\label{E:5.3}
\begin{split}
u(s,z)&\geq
s a_0^{\alpha}\left(\sum_{x\in V}\mu(x)p(s,z,x)a(x)\right)\\
&\geq s a_0^{\alpha}\mu(x_0)a(x_0)p(s,x_0,z) \\
&\geq s a_0^{\alpha}\mu(x_0)a(x_0)p_r(s,x_0,z).
\end{split}
\end{equation}

Substituting \eqref{E:5.3} into \eqref{E:5.2}, we obtain that for $t>0$ and $r>0$,
\begin{equation}\label{eq56}
\begin{split}
G(t,r)&\geq \big(a_0^{\alpha}C_1\big)^{1+\alpha}\int_0^t s^{1+\alpha}
\left(\sum_{z\in B_r}\mu(z)p_r(2t-s,x_0,z)p_r(s,x_0,z)\right)^{1+\alpha} ds\\
&= \big(a_0^{\alpha}C_1\big)^{1+\alpha}p_r(2t,x_0,x_0)^{1+\alpha}\int_0^t s^{1+\alpha} ds\\
&= \frac{\big(a_0^{\alpha}C_1\big)^{1+\alpha}}{2+\alpha}t^{2+\alpha}p_r(2t,x_0,x_0)^{1+\alpha},
\end{split}
\end{equation}
where $C_1=\mu(x_0)a(x_0)$.

By Proposition \ref{estimate} and Definition \ref{heat kernel},
there exists a constant $R'\geq\sqrt{\frac{\rho}{c_3}}$, which is
independent of $t$, such that for $r>R'$,
\begin{equation*}
\begin{split}
p_r(2t,x_0,x_0)\geq \frac{c_2}{\mathcal{V}(x_0,\sqrt{2t})}
\quad\quad (t>t_0),
\end{split}
\end{equation*}
which, along with $\mathcal{V}(x_0,t)\leq c''t^m \log^\eta t$,
yields that
\begin{equation*}
\begin{split}
p_r(2t,x_0,x_0)\geq \frac{2^\eta c_2}{c''}(2t)^{-\frac{m}{2}}\left(\log 2t\right)^{-\eta} \quad\quad (t>\rho).
\end{split}
\end{equation*}

Applying the above inequality to \eqref{eq56},  we obtain that for
$t>\rho$ and $r>R'$,
\begin{equation}\label{eq57}
\begin{split}
G(t,r)&\geq \frac{2^{\eta(1+\alpha)-\frac{m(1+\alpha)}{2}}}{2+\alpha}\left(\frac{a_0^{\alpha}c_2C_1}{c''}\right)^{1+\alpha}
t^{2+\alpha-\frac{m(1+\alpha)}{2}}
\left(\log 2t\right)^{-\eta(1+\alpha)}.
\end{split}
\end{equation}

Combining \eqref{E:5.1-1} and \eqref{eq57}, we get
\begin{equation*}
\begin{split}
\frac{2^{\eta(1+\alpha)-\frac{m(1+\alpha)}{2}}}{2+\alpha}\left(\frac{a_0^{\alpha}c_2C_1}{c''}\right)^{1+\alpha}
\frac{t^{2+\alpha-\frac{m\alpha}{2}}}{\left(\log 2t\right)^{\eta(1+\alpha)}\left(\log t\right)^{\zeta}}
\leq
\frac{c_1C''}{2^\zeta c'}\left(\log\sqrt{c_3r^2}\right)^\eta.
\end{split}
\end{equation*}
In virtue of $m\alpha\leq 2$,  we obtain
\begin{equation*}
\begin{split}
\widetilde{C} \frac{t^{1+\alpha}}{\left(\log
2t\right)^{\eta(1+\alpha)}\left(\log t\right)^{\zeta}} \leq 1,
\end{split}
\end{equation*}
this yields
\begin{equation}
\label{E:5.4}
\begin{split}
\widetilde{C} \left(\frac{t^b}{\log
2t}\right)^{\eta(1+\alpha)+\zeta}
\leq 1,
\end{split}
\end{equation}
for $t>\rho$, where $b=\frac{1+\alpha}{\eta(1+\alpha)+\zeta}>0$,
$\widetilde{C}=\frac{c'2^{\left(\eta-\frac{1}{\alpha}\right)(1+\alpha)+\zeta}}{(2+\alpha)c_1C''}
\left(\frac{a_0^{\alpha}c_2C_1}{c''}\right)^{1+\alpha}\left(\log\sqrt{c_3r^2}\right)^{-\eta}$
and $r>R'$.

Actually, it is easy to see that the reverse and strict inequality
of \eqref{E:5.4} holds for sufficiently large $t$. This leads to a
contradiction. Hence, the equation \eqref{E:1.1} has no
non-negative global solution. This completes the proof of Theorem
\ref{theorem3}.
\end{proof}

\section{Concluding remarks}

Let us consider a special case of Theorem \ref{theorem3}. Taking $\zeta=0$ in Theorem \ref{theorem3},
we have
\begin{claim}
\label{claim3}
\textnormal{
Suppose that $\inf_{x\in V}a(x)>0$ and $D_\mu,D_\omega<\infty$. Let $G$ satisfy
$CDE'(n,0)$ and $c'r^m\leq \mathcal{V}(x_0,r)\leq c''r^m \log^\eta r$ $(\eta>0)$ for $r\geq r_0>1$ and $x_0\in V$. If $0<\alpha\leq
\frac{2}{m}$, then all non-negative solutions of \eqref{E:1.1}
blow up in a finite time.}
\end{claim}

Recently, we have proved in \cite{LW2} the following
\begin{claim}
\label{claim2}
\textnormal{
Under the assumption of Claim \ref{claim3},
if $\alpha>\frac{2}{m}$, then there exists a non-negative global solution to \eqref{E:1.1} when the initial value satisfies
$a(x)\leq \delta p(\gamma, x_0,x)$ for some small $\delta>0$ and any fixed $\gamma\geq r_0^2$.}
\end{claim}

The results, described in Claims \ref{claim3} and \ref{claim2}, provide
a complete answer to the existence problem of non-negative global solutions for equation \eqref{E:1.1} with $\inf_{x\in V}a(x)>0$
on locally finite graphs that satisfy $CDE'(n,0)$ and $c'r^m\leq \mathcal{V}(x_0,r)\leq c''r^m \log^\eta r$ $(\eta>0)$ for $r\geq r_0$.

Besides, under the condition that $D_\mu, D_\omega<\infty$ and the
graph satisfies $CDE'(n,0)$ and uniform polynomial volume growth
of degree $m$, we have obtained in \cite{LW2} the following
result:
\begin{claim}
\label{claim1}
\textnormal{
If $0<\alpha<\frac{2}{m}$, then all non-negative solutions of \eqref{E:1.1} blow up in a finite time and
if $\alpha>\frac{2}{m}$, then there exists a non-negative global solution to \eqref{E:1.1} for a sufficiently small initial value.}
\end{claim}

As a complement to Claim \ref{claim1}, the present result stated
in Theorem \ref{theorem2} reveals that all non-negative solutions
of \eqref{E:1.1} blow up in a finite time if $\alpha=\frac{2}{m}$.


\section*{Acknowledgments}

This research was supported by the National Science Foundation of China (No.11901550).

\vspace{30pt}

\def\refname{\Large

\vspace{50pt}

{\setlength{\parindent}{0pt}
Yiting Wu,\\
Department of Mathematics, China Jiliang University,  Hangzhou, 310018, P. R. China\\
yitingwu@cjlu.edu.cn; yitingly@126.com}

\end{document}